\documentclass[MSNbibl,number,citesort,dvips]{arxbj}
\usepackage{upgreek}
\usepackage{graphicx}

\aid{0}
\volume{19}
\issue{5A}
\pubyear{2013}
\firstpage{1535}
\lastpage{1558}
\doi{10.3150/12-BEJ418} 

\makeatletter
\newtheorem{lemma}{Lemma}
\newtheorem{corollary}{Corollary}
\newtheorem{proposition}{Proposition}
\newproclaim{definition}{Definition}
\newproclaim{assumption}{Assumption}
\newremark{example}{Example}

\newcommand{\bzero}{\boldsymbol{0}}
\newcommand{\bh}{\mathbf{h}}
\newcommand{\br}{\mathbf{r}}
\newcommand{\bs}{\mathbf{s}}
\newcommand{\by}{\mathbf{y}}
\newcommand{\cL}{\mathcal{L}}
\newcommand{\cS}{\mathcal{S}}
\newcommand{\hv}{\hat{v}}
\newcommand{\hc}{\hat{c}}
\newcommand{\dint}{\int\!\!\!\int}

\newcommand{\ag}{\alpha}
\newcommand{\ga}{\gamma}
\newcommand{\Gg}{\Gamma}
\newcommand{\dg}{\delta}
\newcommand{\eg}{\varepsilon}
\newcommand{\kg}{\kappa}
\newcommand{\la}{\lambda}
\newcommand{\sg}{\sigma}
\newcommand{\zg}{\zeta}
\newcommand\refeq[1]{(\ref{e:#1})}
\newcommand{\eqref}[1]{(\ref{#1})}
\makeatother

\begin{document}
\begin{frontmatter}

\title{Consistency of the mean and the principal components of
spatially distributed functional data}
\runtitle{Consistency of mean and PCs for spatial functional data}

\begin{aug}
\author[1]{\fnms{Siegfried} \snm{H\"ormann}\corref{}\thanksref{1}\ead[label=e1]{shormann@ulb.ac.be}} \and
\author[2]{\fnms{Piotr} \snm{Kokoszka}\thanksref{2}\ead[label=e2]{Piotr.Kokoszka@colostate.edu}}
\runauthor{S. H\"ormann and P. Kokoszka} 
\address[1]{D\'epartement de Math\'emathique, Universit\'e Libre de
Bruxelles, CP 210,
Bd du Triomphe,
\mbox{B-1050} Brussels, Belgium. \printead{e1}}
\address[2]{Department of Statistics,
Colorado State University,
Fort Collins, CO 80523-1877, USA.\\ \printead{e2}}
\end{aug}

\received{\smonth{1} \syear{2011}}
\revised{\smonth{10} \syear{2011}}

%
\begin{abstract}
This paper develops a framework for the estimation of the functional
mean and the functional principal components when the functions form
a random field. More specifically, the data we study consist of
curves $X(\mathbf{s}_k; t), t \in[0, T]$, observed at spatial points
$\mathbf{s}_1, \mathbf{s}_2, \ldots, \mathbf{s}_N$. We establish
conditions for the
sample average (in space) of the $X(\mathbf{s}_k)$ to be a consistent
estimator of the population mean function, and for the usual
empirical covariance operator to be a consistent estimator of the
population covariance operator. These conditions involve an
interplay of the assumptions on an appropriately defined dependence
between the functions $X(\mathbf{s}_k)$ and the assumptions on the spatial
distribution of the points $\mathbf{s}_k$. The rates of convergence
may be
the same as for i.i.d. functional samples, but generally depend on the
strength of dependence and appropriately quantified distances
between the points $\mathbf{s}_k$. We also formulate conditions for the
lack of consistency.
\end{abstract}

%

\end{frontmatter}

\section{Introduction} \label{s:i}
This paper develops aspects of theory for
functional data observed at spatial locations.
The data consist of curves $X(\bs_k)=\{X(\bs_k; t), t \in[0,
T]\}$, observed at spatial points $\bs_1, \bs_2, \ldots, \bs_N$. Such
data structures are quite common, but often the spatial
dependence and the spatial distribution of the points $\bs_k$ are not
taken into account. A well-known example is the Canadian temperature
and precipitation data used in
Ramsay and Silverman \cite{ramsay:silverman:2005}. The annual curves are available
at 35 locations, some of which are quite close, and so the curves have
very similar characteristics, others are very remote with notably
different curves.

There has not been much research on fundamental properties of
spatially distributed functional data.
Delicado  \textit{et al.} \cite{delicado:2010} review recent
contributions to the methodology for spatially distributed functional
data. For geostatistical functional data, several exploratory
approaches to kriging have been proposed. Typically fixed basis
expansions are used, see Yamanishi and
Tanaka \cite{yamanishi:tanaka:2003} and
Bel  \textit{et al.} \cite{bel:2011}.
A general theoretical framework has to address several problems. The
first issue is the dimensionality of the index space. While in time
series analysis, the process is indexed by an equispaced scalar
parameter, we need here a $d$-dimensional index space. For model
building this makes a big difference since the dynamics and dependence
of the process have to be described in all directions, and the typical
recurrence equations used in time series cannot be employed. The model
building is further complicated by the fact that the index space is
often continuous (geostatistical data). Rather than defining a random
field $\{\xi(\bs);\bs\in\mathbb{R}^d\}$ via a specific model
equations, dependence conditions are imposed, in terms of the decay of
the covariances or using mixing conditions. Another feature peculiar
to random field theory is the design of the sampling points; the
distances between them play a fundamental role. Different asymptotics
hold in the presense of clusters and for sparsely distributed points.
At least three types of point distributions have been considered
(Cressie \cite{cressie:1993}): When the region $R_N$ where the points
$\{\bs_{i,N};1\leq i \leq N\}$ are sampled remains bounded, then we
are in the so-called \emph{infill domain sampling} case. Classical
asymptotic results, like the law of large numbers or the central limit
theorem will usually fail, see Lahiri \cite{lahiri:1996}. The other
extreme situation is described by the \emph{increasing domain
sampling}. Here a minimum separation between the sampling points
$\{\bs_{i,N}\}\in R_N$ for all $i$ and $N$ is required. This is of
course only possible if $\operatorname{diam}(R_N) \to\infty$. We shall
also explore the \emph{nearly infill} situation studied by
Lahiri \cite{lahiri:2003} and Park \textit{et al.} \cite{park:kim:park:hwang:2009}. In
this case, the domain of the sampling region becomes unbounded
($\operatorname{diam}(R_N)\to\infty$), but at the same time the
number of
sites in any given subregion tends to infinity, that is, the points
become more dense. These issues are also studied by
Zhang \cite{zhang:2004}, Loh \cite{loh:2005},
Lahiri and Zhu \cite{lahiri:zhu:2006}, Du  \textit{et al.} \cite{du:zhang:mandrekar:2009}. We
formalize these concepts in Section~\ref{s:me}.
Finally, the interplay of the geostatistical spatial structure and the
functional temporal structure must be cast into a workable framework.

The paper is organized as follows. Section \ref{s:me} introduces the
statistical setting. It also compares our conditions
to those typically assumed for scalar spatial processes. In Sections
\ref{s:cm} and \ref{s:cc} we establish consistency results,
respectively, for the functional mean and the covariance operator.
Section \ref{s:example} explains, by means of general theorems and
examples, when the sample principal components are not consistent. The
proofs of the main results are collected in Section \ref{s:p-cm}.

To make this presentation more streamlined, we have outsourced
some proofs, further details and several examples. They are available
as supplemental
material: H\"ormann and
Kokoszka \cite{suppl:hormann:kokoszka:2011}.
%
\section{Preliminaries and assumptions}
\label{s:me}
We assume $\{ X(\bs), \bs\in\mathbb{R}^d\}$ is a
random field taking values in $L^2=L^2([0,1])$, that is, each
$X(\bs)$ is a square integrable function defined on $[0,1]$.
The value of this function at $t\in[0,1]$ is denoted by
$X(\bs; t)$. With the usual inner product in $L^2$, the norm
of $X(\bs)$ is
\[
\|X(\bs)\| = \biggl\{\int X^2(\bs; t)\,\mathrm{d}t\biggr\}^{1/2}.
\]
The mean function $\mu(\bs)=\{EX(\bs;t), t\in[0,1]\}$
and the covariance operator is
then defined for $x\in L^2$ by
\[
C_{\bs,\bs}(x)= E\bigl[ \langle X(\bs)-
\mu(\bs) , x \rangle\bigl(X(\bs)-\mu(\bs)\bigr)\bigr].
\]
More generally, we define the cross-covariance operators
\[
C_{\bs_1,\bs_2}(x)= E\bigl[ \langle X(\bs_1)-
\mu(\bs_1) , x \rangle\bigl(X(\bs_2)-\mu(\bs_2)\bigr)\bigr].
\]
For the existence of $C_{\bs_1,\bs_2}$ a minimal assumption is that
the variables have finite second moments in the sense that
%
\begin{equation}\label{e:si}
E\|X(\bs)\|^2<\infty\qquad \forall\bs.
\end{equation}

To think of our observations $X(\bs)$ as curves in $L^2$ is convenient
and motivated this work, but our propositions in Sections~\ref{s:cm},
\ref{s:cc} and
\ref{s:example} only require the general assumption
that $\{ X(\bs), \bs\in\mathbb{R}^d\}$
is a field taking values in some separable Hilbert space. So particularly
our results hold true when $L^2$ is replaced by $\mathbb{R}^p$. To
to the best of our knowledge, our results are new as well in the vector case.

Our goal is to estimate the unknown mean curves and the principal
components (FPC's). FPC's are intimately connected with covariance operators,
as we will describe in some detail in the next section, and likewise estimation
of FPC's is based on estimation of covariance operators.
Thereby, we are estimating \emph{across
space} and \emph{not across time.} A minimal requirement for this to
make sense is
then that all locations share a common mean curve and that the
covariance operator is the same for all locations, respectively:
%
\begin{equation}\label{e:stat}
\mu(\bs)=\mu\quad\mbox{and}\quad C_{\bs,\bs}=C.
\end{equation}

Although \eqref{e:stat} is apparently necessary and would also be
quite natural if we were in time series context, it may be not
realistic in some spatial data situations. Let us briefly sketch
how our methods can still be useful by
employing a spatio-temporal framework. In this case we
suppose to have for each location $\bs$ a
functional time series $\{X_i(\bs),i=1,2,\ldots\}$. To avoid confusion
between the time parameter $i$ ($i=1,2,3,\ldots$) and the ``intraday
time'' parameter $t$ ($t\in[0,1]$),
we will employ for this paragraph the notation $X_i(\bs;t)$.
We then assume the following model
\[
X_i(\bs;t)=\alpha(\bs;t)+\mu_i(t)+Y_i(\bs;t),\qquad i=1,2,3,\ldots,
\]
where $\{Y_i(\bs;t), t\in[0,1]\}$ are curves at spatial locations which
satisfy \refeq{stat} with $\mu\equiv0$ and where for each
fixed $\bs$ the functional time series $\{X_i(\bs;\cdot),
i=1,2,3,\ldots\}$ is stationary
and weakly dependent (e.g., as assumed in H{\"o}rmann and
Kokoszka \cite{hormann:kokoszka:2010}).
The random curves $\{\mu_i(t), t\in[0,1]\}$ are zero mean curves and
form a certain ``basis level''
shared by all curves $\{X_i(\bs;t), t\in[0,1]\}$ across space at time
$i$. In our setup,
the problem of interest is to determine
for a given day $i$ the curve $\{\mu_i(t), t\in[0,1]\}$. We might
think of daily
temperature curves measured across some region.
The curve $\{\mu_i(t), t\in[0,1]\}$ amounts to a general
deviation from normal on day $i$ which persists over the whole region
(e.g., on
a hot day, it is not unlikely that all stations will show
curves above average). Due to geographical differences
the mean curves $\alpha(\bs;t)$ might be different at different locations,
but they can be estimated
individually by $\hat\alpha(\bs;t)=\frac{1}{M}\sum
_{i=1}^MX_{i}(\bs;t)$, when we have a sample
of $M$ days.
So we can assume that we can
detrended our data and then work with new observations $X_i'(\bs
;t)=\mu_i(t)+Y_i(\bs;t)$,
to which our theory applies, when now $i$ is fixed.

Besides \refeq{stat}, we don't impose any further stationarity
assumption on the field
$\{X(\bs)\}$.
%
%
%
%
\subsection{Functional principal components}
Functional principal components play a fundamental role in functional
data analysis, much greater than the usual multivariate principal
components. This is mostly due to the fact that the
Karhunen--Lo\`eve expansion allows to represent functional data
in a concise way. This property has been extensively used and studied
in various settings.
To name only a few illustrative references, we cite
Hall and
Hosseini-Nasab \cite{hall:h-n:2006},
Reiss and Ogden \cite{reiss:ogden:2007},
Gabrys and
Kokoszka \cite{gabrys:kokoszka:2007},
Benko \textit{et al.} \cite{benko:hardle:kneip:2009},
Paul and Peng \cite{paul:peng:2009},
Jiang and Wang \cite{jiang:wang:2010}
and Gabrys \textit{et al.} \cite{gabrys:horvath:kokoszka:2010}.
Depending on the structure of the data, theoretical analyses
emphasize various aspects of the estimation process, with smoothing
in i.i.d. samples having being particularly carefully studied.
This paper focuses on the spatial dependence and distribution of the
curves, which has received no attention so far.

Suppose $X_1, X_2, \ldots, X_N$ are mean zero identically
distributed elements of $L^2$
such that $E\|X\|^4 < \infty$.
The eigenfunctions of the covariance operator of $X_1$ are the
functional principal components,
denoted $v_k$.
Up to a sign, they are estimated by the empirical FPC's (EFPC's), denoted
$\hv_k$ and defined as the eigenfunctions of the empirical covariance
operator
\[
\widehat C_N(x) = \frac{1}{N} \sum_{n=1}^N \langle X_n, x \rangle X_n,\qquad
x \in L^2.
\]

The distance between $v_k$ and $\hv_k$ is determined by the distance
between $C$ and $\widehat C_N$. This follows from Lemma \ref{l:B4.3}.
Due to its importance in the exposition that follows, we provide a
precise statement. An analog of Lemma \ref{l:B4.3} has often been used
for i.i.d. functions under more restrictive assumptions. To state Lemma
\ref{l:B4.3}, consider two compact operators $C$ and $K$ with singular
value decompositions
%
\begin{equation}\label{e:C-K} C(x) = \sum_{j=1}^\infty\la_j
\langle x, v_j \rangle f_j, \qquad K(x) = \sum_{j=1}^\infty\ga_j \langle x, u_j
\rangle g_j.
\end{equation}
Recall that a linear operator $K$ in a separable
Hilbert space $H$ is said to be Hilbert--Schmidt, if for some
orthonormal basis $\{e_i\}$ of $H$
\[
\|K\|_{\cS}^2:=\sum_{i\geq1}\|K(e_i)\|_{H}^2<\infty.
\]
Then $\|\cdot\|_{\cS}$ defines a norm on the space of all operators
satisfying
this condition. The norm is independent of the choice of the basis.
This space is again a Hilbert space with the inner product
\[
\langle K_1, K_2\rangle_{\cS}
=\sum_{ i\geq1}\langle K_1(e_i), K_2(e_i)\rangle.
\]

Set
\[
v_j^\prime=\hc_j v_j,\qquad \hc_j=\operatorname{sign}(\langle u_j,v_j
\rangle).
\]
\begin{lemma} \label{l:B4.3}
Suppose $C, K \in\cL$ are two compact operators with singular
value decompositions~\refeq{C-K}.
If $C$ is symmetric, $f_j=v_j$ in \eqref{e:C-K},
and its eigenvalues satisfy
%
\begin{equation}\label{e:la>d}
\la_1 > \la_2> \cdots> \la_d > \la_{d+1},
\end{equation}
then
\[
\|u_j - v_j^\prime\| \le\frac{2\sqrt{2}}{\ag_j} \|K-C\|_\cL,\qquad 1
\le j\le p,
\]
where $\ag_1 = \la_1- \la_2$ and
$
\ag_j = \min( \la_{j-1} - \la_j, \la_j - \la_{j+1}), 2 \le j \le p.
$
\end{lemma}

Lemma \ref{l:B4.3} can be proven using Corollary 1.6 on page 99
of Gohberg \textit{et al.} \cite{gohberg:1990} and following the lines of
the proof of Lemma 4.3 of Bosq \cite{bosq:2000}.

If the functional observations $X_k, k=1,2, \ldots$\,,
are independent, then
%
\begin{equation}\label{e:standard-rate}
\limsup_{N\to\infty} N E\| \widehat C_N - C\|_{\cS}^2 < \infty.
\end{equation}
Consequently, for such functional observations, under \refeq{la>d},
\[
\max_{1 \le k \le d} E\|\hc_k \hv_k - v_k\|^2 = \mathrm{O}( N^{-1}).
\]

H{\"o}rmann and
Kokoszka \cite{hormann:kokoszka:2010} showed that \refeq{standard-rate}
continues to hold for weakly dependent time series, in particular
for $m$-dependent $X_k$.
It is not difficult to show that if spatially distributed functions
are such that $X(\bs)$ is independent of $X(\bs^\prime)$ if the
distance between $\bs$ and $\bs^\prime$
is greater than $m$, then \refeq{standard-rate} need not hold.
It is even possible that
EFPC's $\hv_k$ do not converge at all. See Section~\ref{s:example}.
\subsection{Dependence assumptions}
To develop an estimation framework, we impose
conditions on the decay of the
cross--covariances $E[\langle X(\bs_1)-\mu, X(\bs_2)-\mu\rangle]$,
as the distance between $\bs_1$ and $\bs_2$ increases. We
shall use the distance function defined by the Euclidian norm in
$\mathbb{R}^d$, denoted $\|\bs_1-\bs_2\|_2$,
but other distance functions can be used as well.
\begin{assumption}\label{ass:dep}
The spatial process $\{X(\bs), \bs\in\mathbb{R}^d\}$ 
satisfies \refeq{si} and \refeq{stat}. In addition,
%
\begin{equation}\label{eq:dep}
| E \langle X(\bs_1)-\mu, X(\bs_2)-\mu\rangle|
\leq h (\|\bs_1-\bs_2\|_2 ),
\end{equation}
where $h\dvtx[0,\infty)\to[0,\infty)$ with $h(x)\searrow0$, as $x\to
\infty$.
\end{assumption}

If $\{e_j\}$ is an orthonormal basis (ONB) of $L^2$, then it can be
easily seen
that \eqref{eq:dep} is equivalent to
%
\begin{equation}\label{eq:dep1}
\biggl|\sum_{j\geq1}\langle C_{\bs_1,\bs_2}(e_j), e_j\rangle\biggr|
\leq h (\|\bs_1-\bs_2\|_2 ).
\end{equation}
For any such ONB a Fourier expansion of $X(\bs)$ yields
%
\begin{equation}\label{e:xi}
X(\bs;t) = \mu+\sum_{j=1}^\infty\xi_j(\bs) e_j(t),\qquad
\bs\in\mathbb{R}^d, t \in[0,1],
\end{equation}
where $\xi_j(\bs)=\langle X(\bs)-\mu,e_j\rangle$.
Using the relation
\[
\langle C_{\bs_1,\bs_2}(e_j), e_j\rangle
=E [\xi_j(\bs_1)\xi_j(\bs_2) ],
\]
the more specifical assumption
%
\begin{equation}\label{eq:corr}
|E [\xi_j(\bs_1)\xi_j(\bs_2) ] |\leq\phi_j(\|\bs_1-\bs_2\|_2),
\end{equation}
on the scalar fields, gives \eqref{eq:dep}, if
%
\begin{equation}\label{e:dep-j}
\sum_{j\geq1}\phi_j (\|\bs_1-\bs_2\|_2 )
\leq h (\|\bs_1-\bs_2\|_2 ).
\end{equation}

Examples \ref{ex:mixing} and \ref{ex:s-decay} consider typical spatial
covariance functions, and show when condition \refeq{dep-j}
holds with a function $h$ as in Assumption \ref{ass:dep}.
\begin{example}\label{ex:mixing}
Suppose that the fields $\{\xi_j(\bs), \bs\in\mathbb{R}^d\}$,
$j\geq1$, are zero mean, strictly stationary
and \emph{$\alpha$-mixing.}
That is
\[
\sup_{(A,B)\in\sigma(\xi_j(\bs))\times\sigma(\xi_j(\bs+\bh
))}|P(A)P(B)-P(A\cap B)|\leq
\alpha_j(\bh),
\]
with $\alpha_j(\bh)\to0$ if $\|\bh\|_2\to\infty$. Let $\alpha
_j'(h)=\sup\{\alpha_j(\bh)\dvt
\|\bh\|_2=h\}$. Then
$\alpha_j^*(h)=\sup\{\alpha_j'(x)\dvt\allowbreak
x\geq h\}\searrow0$ as $h\to\infty$.\vadjust{\goodbreak}
Using stationarity and the main result in Rio \cite{rio:1993}
it follows that
\begin{eqnarray*}
|E[\xi_j(\bs_1)\xi_j(\bs_2)]|&=&|E[\xi_j(\bzero)\xi_j(\bs_2-\bs
_1)]|\\
&\leq&2\int_0^{2\alpha_j(\bs_2-\bs_1)}Q_{j}^2(u)\,\mathrm{d}u\\
&\leq&2\int_0^{2\alpha_j^*(\|\bs_2-\bs_1\|_2)}Q_{j}^2(u)\,\mathrm{d}u\\
&{\hspace*{2.8pt}=:}&\phi_j(\|\bs_2-\bs_1\|_2),
\end{eqnarray*}
where $Q_{j}(u)=\inf\{t\dvt P(|\xi_j(\bzero)|>t)\leq u\}$ is the
quantile function of $|\xi_j(\bzero)|$. Note that $\alpha_h(\bh
)\leq1/4$ for any $\bh$, and thus
$\phi_j(x)\leq2\int_0^1Q_{j}^2(u)\,\mathrm{d}u= 2E[\xi_j^2(\bzero)]$. If
$\sum_{j\geq1}E\xi_j^2(\bzero)<\infty$, then \eqref{eq:dep}
holds with $h(x)=\sum_{j\geq1}\phi_j(x)$. (Note that
$|h(x)|\searrow0$ follows
from $\alpha_j^*(x)\searrow0$ and the monotone convergence theorem.)

We note that $\alpha$-mixing is one of the classical assumptions in
random field literature to establish limit theorems. It is in fact a much
stronger assumption than ours and it is suitable if one needs more
delicate results, like a central limit theorem (see, e.g., Bolthausen \cite
{bolthausen:1982}) or uniform laws of large numbers, see Jenish and
Prucha \cite
{jenish:prucha:2009}. Besides the
restriction to scalar observations, many papers restrict to the
so-called ``purely
increasing domain sampling,'' an assumption that we are going to
further relax in the following.
\end{example}
\begin{example} \label{ex:s-decay}
Suppose \eqref{eq:corr} holds, and set $h(x) = \sum_{j \ge1}\phi_j(x)$.
If each $\phi_j$ is a powered exponential covariance function defined
by
\[
\phi_j(x) = \sg_j^2 \exp\biggl\{- \biggl(\frac{x}{\rho_j}\biggr)^p \biggr\}.
\]
Then $h$ satisfies the conditions of Assumption \ref{ass:dep}
if
%
\begin{equation}\label{e:cond-h}
\sum_{j \ge1} \sg_j^2 < \infty\quad
\mbox{and}\quad \sup_{j \ge1} \rho_j < \infty.
\end{equation}
Condition \refeq{cond-h} is also sufficient if all $\phi_j$ are
in the Mat{\'e}rn class, see Stein \cite{stein:1999},
with the same $\nu$, that is,
\[
\phi_j(x) = \sg_j^2 x^\nu K_\nu(x/\rho_j),
\]
because the modified Bessel function $K_\nu$ decays monotonically
and approximately
exponentially fast; numerical calculations show that $K_\nu(s)$
practically vanishes if $s> \nu$. Condition \refeq{cond-h} is
clearly sufficient for spherical $\phi_j$ defined (for $d=3$) by
\[
\phi_j(x) = \cases{
\sg_j^2 \biggl(1 - \displaystyle\frac{3x}{2\rho_j} + \frac{x^3}{2\rho_j^3} \biggr),
&\quad $x \le\rho_j$,\cr
0, &\quad $x > \rho_j$}
\]
because $\phi_j$ is decreasing on $[0, \rho_j]$.
\end{example}

Assumption~\ref{ass:dep} is appropriate when studying estimation of
the mean function. For the estimation of the covariance operator, we
need to impose a different assumption. Recall that if $z$ and $y$ are
elements in some Hilbert space $H$ with norm $\|\cdot\|_H$, the
operator $z\otimes y$, is defined by $z\otimes y(x)=\langle z,
x\rangle y$. In the following assumption, we suppose that the mean of
the functional field is zero. This is justified by notational
convenience and because we deal with the consistent estimation of the
mean function separately.
\begin{assumption}\label{ass:dep-cov}
The spatial process $\{X(\bs), \bs\in\mathbb{R}^d\}$
satisfies \eqref{e:stat} with $\mu\equiv0$ and has 4 moments,
that is, $E\langle X(\bs), x\rangle=0$, $\forall x\in L^2$
and $E\|X(\bs)\|^4<\infty$. In addition,
%
\begin{equation}\label{eq:dep_cov}
|E\langle X(\bs_1)\otimes X(\bs_1)-C ,
X(\bs_2)\otimes X(\bs_2)-C\rangle_{\cS} |
\leq H (\|\bs_1-\bs_2\|_2 ),
\end{equation}
where $H\dvtx[0,\infty)\to[0,\infty)$ with $H(x)\searrow0$, as $x\to
\infty$.
\end{assumption}

Assumption \ref{ass:dep-cov} cannot be verified using only conditions
on the covariances of the scalar fields $\xi_j$ in \eqref{e:xi}
because these covariances do not specify the 4th order structure of
the model. This can be done if the random field is Gaussian or if
additional structure is imposed. We discuss this Assumption~\ref
{ass:dep-cov} in
more detail in supplemental material (H\"ormann and
Kokoszka, \cite{suppl:hormann:kokoszka:2011}, Section~S.2).
%
\subsection{Spatial distribution of the sampling points}
As already noted, for spatial processes assumptions on
the distribution of the sampling points are as important as those on
the covariance structure. To formalize the different sampling schemes
introduced in Section \ref{s:i}, we propose the following measure of
``minimal dispersion'' of some point cloud $\mathfrak{S}$:
\[
I_\rho(\bs,\mathfrak{S})=\bigl|\{\by\in\mathfrak{S}\dvt \|\bs-\by\|
_2\leq\rho\}\bigr|/|
\mathfrak{S}|
\quad\mbox{and}\quad
I_\rho(\mathfrak{S})=\sup\{ I_\rho(\bs,\mathfrak{S}), \bs\in
\mathfrak{S} \},
\]
where $|\mathfrak{S}|$ denotes the number of elements of
$\mathfrak{S}$. The quantity $I_\rho(\mathfrak{S})$ is the maximal
fraction of $\mathfrak{S}$-points in a ball of radius $\rho$ centered
at an element of $\mathfrak{S}$. Notice that $ 1/|\mathfrak{S}| \le
I_\rho(\mathfrak{S})\le1$. We call $\rho\mapsto
I_\rho(\mathfrak{S})$ the \emph{intensity function} of $\mathfrak{S}$.
\begin{definition}\label{def:nids}
For a sampling scheme $\mathfrak{S}_N=\{\bs_{i,N}; 1\leq i\leq S_N\}
$, $S_N\to\infty$,
we consider the following conditions:
\begin{enumerate}[(iii)]
\item[(i)] there is a $\rho>0$ such that
$\limsup_{N\to\infty}I_{\rho}(\mathfrak{S}_N)>0$;
\item[(ii)] for some sequence $\rho_N\to\infty$
we have $I_{\rho_N}(\mathfrak{S}_N)\to0$;
\item[(iii)] for any fixed $\rho>0$ we have
$S_N I_{\rho}(\mathfrak{S}_N)\to\infty$.
\end{enumerate}
We call a deterministic sampling scheme
$\mathfrak{S}_N=\{\bs_{i,N}; 1\leq i\leq S_N\}$

\emph{Type A}: if \textup{(i)} holds;

\emph{Type B}: if \textup{(ii)} and \textup{(iii)} hold;

\emph{Type C}: if \textup{(ii)} holds, but there is a $\rho>0$ such that
$\limsup_{N\to\infty}S_NI_\rho(\mathfrak{S}_N)< \infty$.\\
If the sampling scheme is stochastic we call it Type \emph{A}, \emph{B} or
\emph{C}
if relations \emph{(i)}, \emph{(ii)} and \emph{(iii)} hold with $I_\rho(\mathfrak{S}_N)$
replaced by $EI_\rho(\mathfrak{S}_N)$.
\end{definition}

Type A sampling is related to purely infill domain sampling which
corresponds to $I_{\rho}(\mathfrak{S}_N)=1$ for all $N\geq1$,
provided $\rho$ is large enough. However, in contrast to the purely
infill domain sampling, it still allows for
a non-degenerate asymptotic theory for sparse enough subsamples
(in the sense of Type B or C).
%

A brief reflection shows that assumptions (i) and (ii)
are mutually exclusive. Combining (ii) and (iii) implies that the
points intensify (at least at certain spots) excluding the purely
increasing domain sampling. Hence, the Type B sampling corresponds to
the nearly infill domain sampling. If only (ii) holds, but (iii) does
not (Type C sampling) then the sampling scheme corresponds to purely
increasing domain sampling.

Our conditions are more general than those proposed so far.
%
We treat below two special
cases which are closely related to those considered by
Lahiri \cite{lahiri:2003}. The points are assumed to be on a grid
of an increasing size, or to have a density.
We show how our more
general assumptions look in these special cases, and provide
additional intuition behind the sampling designs formulated in
Definition~\ref{def:nids}. They also set a framework for some
results of Sections \ref{s:cm} and \ref{s:cc}.
%
\subsection{Non-random regular design}\label{se:nrrd}
Let $\mathcal{Z}(\boldsymbol{\delta})$ be a lattice in $\mathbb{R}^d$
with increments $\delta_i$ in the $i$th direction. Let
$\delta_0=\min\{\delta_1,\ldots,\delta_d\}$,
$\Delta^d=\prod_{i=1}^d\delta_i$ and let $R_N=\ag_NR_0$, where $R_0$
is some bounded Riemann measurable
Borel-set in $\mathbb{R}^d$ containing the origin.
A set is Riemann measurable if its indicator function is
Riemann integrable. This condition excludes highly irregular
sets $R_0$. The scaling parameters $\ag_N>0$ are assumed to be
non-decreasing and will be specified below in Lemma~\ref{le:scaling}.
We assume without loss of generality that $\operatorname{Vol}(R_0)=1$, hence
$\operatorname{Vol}(R_N)=\ag_N^d$. Typical examples are
$R_0=\{x\in\mathbb{R}^d\dvt \|x\|\leq z_{1,d}\}$, with $z_{1,d}$ equal
to the radius of the $d$-dimensional sphere with volume 1, or
$R_0=[-1/2,1/2]^d$. The sampling points $\mathfrak{S}_N$ are defined
as $\{\bs_{k,N}, 1\leq k\leq
S_N\}=\mathcal{Z}(\eta_N\boldsymbol{\delta})\cap R_N$, where $\eta_N$
is chosen such that the sample size $S_N\sim N$. We remark
that we only introduce $S_N$ as it is generally not possible
by the just described construction to define $\eta_N$ such, that
we would get exactly $N$ points on the grid to fall in $R_N$. It is intuitively
clear that $\operatorname{Vol}(R_N)\approx\eta_N^d\Delta^d S_N$,
suggesting
%
\begin{equation}\label{eta}
\eta_N=\frac{\ag_N}{\Delta N^{1/d}}.
\end{equation}
A formal proof that $\eta_N$ in \eqref{eta} assures $S_N\sim N$
is given in Section S.3 of the supplemental material.

The following lemma, whose proof
is also given in the supplemental material
relates the non-random regular design to
Definition~\ref{def:nids}.
We write $a_N\gg b_N$ if $\limsup b_N/a_N<\infty$.

\begin{lemma}\label{le:scaling} In the above
described design, the following pairs of statements
are equivalent:
\begin{enumerate}[(iii)]
\item[(i)] $\ag_N$ remains bounded $\Leftrightarrow$ Type \emph{A} sampling;
\item[(ii)] $\ag_N\to\infty$ and
$\ag_N=\mathrm{o}(N^{1/d})$ $\Leftrightarrow$ Type \emph{B} sampling;
\item[(iii)] $\ag_N\gg N^{1/d}$ $\Leftrightarrow$ Type \emph{C} sampling.
\end{enumerate}
\end{lemma}
%
\subsection{Randomized design}\label{se:rd}
Let $\{\bs_{k}, 1\leq k\leq N\}$ be i.i.d. random vectors with a density
$f(\bs)$ which has support on a Borel set $R_0\subset\mathbb{R}^d$
containing the origin and satisfying $\operatorname{Vol}(R_0)=1$.
Again we
assume Riemann measurability for $R_0$ to exclude highly irregular
sets. For the sake of simplicity, we shall assume that on $R_0$ the
density is bounded away from zero, so that we have $0<f_L\leq
\inf_{x\in R_0}f(x)$. The point set $\{\bs_{k,N}, 1\leq k\leq N\}$ is
defined by $\bs_{k,N}=\ag_N\bs_{k}$ for $k=1,\ldots,N$. For fixed
$N$, this is equivalent to: $\{\bs_{k,N}, 1\leq k\leq N\}$ is an i.i.d.
sequence on $R_N=\ag_N R_0$ with density $\ag_N^{-d}f(\ag_N^{-1}\bs)$.

We cannot expect to obtain a full analogue of Lemma~\ref{le:scaling}
in the randomized setup. For Type~C sampling, the problem is much
more delicate, and a closer study shows that it is related to the
oscillation behavior of multivariate empirical processes. While
Stute \cite{stute:1984} gives almost sure upper bounds, we would need
here sharp results on the moments of the modulus of continuity of
multivariate empirical process. Such results exist, see
Einmahl and
Ruymgaart \cite{einmahl:ruymgaart:1987}, but are connected to technical
assumptions on the bandwidth for the modulus (here determined by
$\alpha_N$) which are not satisfied in our setup. Hence, a detailed
treatment would go beyond the scope of this paper. We thus state here
the following lemma whose proof
is given in the supplemental material.
\begin{lemma}\label{le:scaling_random} In the above
described sampling scheme the following statements
hold:
\begin{enumerate}[(ii)]
\item[(i)] $\ag_N$ remains bounded $\Rightarrow$ Type \emph{A} sampling;
\item[(ii)] $\ag_N\to\infty$ and
$\ag_N=\mathrm{o}(N^{1/d})$ $\Rightarrow$ Type \emph{B} sampling.
\end{enumerate}
\end{lemma}
%
\section{Consistency of the sample mean function}
\label{s:cm}
We start with a general setup, and show
that the rates can be improved in special cases.
The proofs of the main results,
Propositions \ref{p:gnis}, \ref{pr:mean_reg}, \ref{pr:mean_rand},
are collected in Section \ref{s:p-cm}.

For independent or weakly dependent functional observations
$X_k$,
%
\begin{equation}\label{e:iid-rate}
E \Biggl\|\frac{1}{N}\sum_{k=1}^N X_k -\mu\Biggr\|^2
= \mathrm{O}( N^{-1}).
\end{equation}
Proposition \ref{p:gnis} shows that for general functional spatial
processes, the rate of consistency may be much slower than
$\mathrm{O}( N^{-1})$; it is the maximum of
$h(\rho_N)$ and $I_{\rho_N}(\mathfrak{S}_N)$ with $\rho_N$ from
(ii) of Definition~\ref{def:nids}. Intuitively, the
sample mean is consistent if there is a sequence of increasing
balls which contain a fraction of points which tends
to zero, and the decay of the correlations compensates for the
increasing radius of these balls.
\begin{proposition}\label{p:gnis}
Let Assumption~\ref{ass:dep} hold, and assume that $\mathfrak{S}_N$ defines
a \emph{non-random} design of Type \emph{A}, \emph{B} or \emph{C}. Then for any $\rho_N>0$,
%
\begin{equation}\label{e:genbound}
E \Biggl\|\frac{1}{N}\sum_{k=1}^NX(\bs_{k,N})-\mu\Biggr\|^2
\leq h(\rho_N)+h(0)I_{\rho_N}(\mathfrak{S}_N).
\end{equation}
Hence, under the Type \emph{B} or Type \emph{C} non-random sampling, with $\rho_N$
as in \emph{(ii)}
of Definition~\ref{def:nids}, the sample mean is consistent.
\end{proposition}
\begin{example}\label{ex:p51}
Assume that $N$ points $\{\bs_{k,N}, 1\leq k\leq N\}$
are on a regular grid in $\ag_N [-1/2,1/2]^d$.
Then, as we have seen in Section~\ref{se:nrrd},
$I_{\rho}(\mathfrak{S}_N)$
is proportional to $(\rho/\ag_N)^d$.

For example, if $h(x)=1/(1+x)^2$, then choosing
$\rho_N=\ag_N^{d/(d+2)}$ we obtain that
\[
h(\rho_N)+h(0)I_{\rho_N}(\mathfrak{S}_N)\ll\ag_N^{-2d/(d+2)}\vee N^{-1}.
\]
(Recall that $I_{\rho_N}(\mathfrak{S}_N)\geq N^{-1}$.)
\end{example}

In H\"ormann and
Kokoszka \cite{suppl:hormann:kokoszka:2011}, Section S.4, we show
that bound \eqref{e:genbound} is optimal, in the sense that
under the assumptions of Proposition~\ref{p:gnis} it is always
possible to
find an example where the rate in \eqref{e:genbound} is precise and cannot
be improved.
This of course doesn't mean that the obtained rate is uniformly optimal.
If we impose a more regular sampling design, we can get better rates.
\begin{proposition}\label{pr:mean_reg}
Assume the sampling design of Section~\ref{se:nrrd}.
Let Assumption~\ref{ass:dep} hold with $h$ such that
$x^{d-1}h(x)$ is monotone on $[b,\infty)$, $b>0$. Then under Type
\emph{B}
sampling\vspace*{-0.67pt}
%
\begin{eqnarray}\label{eq:mean_reg}
&&E \Biggl\|\frac{1}{S_N}\sum_{k=1}^{S_N}X(\bs_{k,N})-\mu\Biggr\|^2\nonumber
\\[-8pt]
\\[-8pt]
&&\quad \leq\frac{1}{\ag_N^d}
\biggl\{ d(3\Delta)^d \int_{0}^{K\ag_N}x^{d-1} h (x)\,\mathrm{d}x
+\mathrm{o}(1) \sup_{x\in[0,K\ag_N]}x^{d-1}h(x)\biggr\}\nonumber
\end{eqnarray}
for some large enough constant $K$ which is independent of $N$.
Under Type \emph{C} sampling $1/\ag_N^d$ in \eqref{eq:mean_reg}
is replaced by $\mathrm{O} (N^{-1} )$.
\end{proposition}

The technical assumptions on $h$ pose no practical problem, they are
satisfied for all important examples, see Example~\ref{ex:s-decay}.
A common situation is that
$x^{d-1}h(x)$ is increasing on $[0,b]$ and decreasing
thereafter. We recall again that $S_N\sim N$ and that $S_N=N$ usually
cannot be achieved due to the construction of the sampling design.

Our next example shows that under nearly infill domain sampling the
rate of consistency may be much slower than for the i.i.d. case,
if the size of the domain does not increase fast enough.
\begin{example} \label{ex:exp-h}
Suppose the functional spatial process has representation \eqref{e:xi},
and \eqref{eq:corr} holds with the covariance functions
$\phi_j$ as in Example \ref{ex:s-decay} (powered exponential,
Mat{\'e}rn or spherical). Define $h(x) = \sum_{j\ge1} \phi_j(x)$,
and assume that condition \refeq{cond-h} holds.
Assumption~\ref{ass:dep} is then satisfied and
%
\begin{equation}\label{e:reg-cond}
\int_{0}^\infty x^{d-1}h(x)\,\mathrm{d}x<\infty\quad\mbox{and}\quad \sup_{x\in\mathbb
{R}}x^{d-1}h(x)<\infty.
\end{equation}
Therefore, for the sampling design of Section \ref{se:nrrd},
%
\begin{equation}\label{e:reg-rate}
E \Biggl\|\frac{1}{S_N}\sum_{k=1}^{S_N}X(\bs_{k,N})-\mu\Biggr\|^2
= \cases{
\mathrm{O}(\ag_N^{-d} \vee N^{-1}), & \quad under Type B sampling,\cr
\mathrm{O}( N^{-1}), & \quad under Type C sampling.}
\end{equation}
\end{example}

The next example shows that formula \refeq{reg-rate} is far
from universal, and that the rate of consistency may be even slower
if the covariances decay slower than exponential.
\begin{example} \label{ex:guad-h}
Consider the general setting of Example \ref{ex:exp-h}, but assume that
each covariance function $\phi_j$ has the quadratic rational form
\[
\phi_j(x) = \sg_j^2 \biggl\{1 + \biggl(\frac{x}{\rho_j}\biggr)^2 \biggr\}^{-1}.
\]
Condition \refeq{cond-h} implies that $h(x) = \sum_{j\ge1} \phi_j(x)$
satisfies Assumption \ref{ass:dep}, but now $h(x) \sim x^{-2}$,
as $x\to\infty$. Because of this rate, condition \refeq{reg-cond}
holds only for $d=1$ (and so for this dimension \refeq{reg-rate}
also holds). If $d\ge2$, \refeq{reg-cond} fails,
and to find the rate of the consistency, we must
use \eqref{eq:mean_reg} directly. We focus only on Type B sampling,
and assume implicitly that the rate is slower than $N^{-1}$.
We assume \refeq{cond-h} throughout this example.

If $d=2$,
\begin{eqnarray*}
\int_{0}^{K\ag_N}x^{d-1} h(x)\,\mathrm{d}x
&=& \sum_j \sg_j^2
\int_{0}^{K\ag_N} x \biggl\{1 + \biggl(\frac{x}{\rho_j}\biggr)^2 \biggr\}^{-1} \,\mathrm{d}x
\\
& =& \sum_j \sg_j^2\rho_j^2 \mathrm{O}\biggl(\int_{1}^{K\ag_N} x^{-1}\,\mathrm{d}x\biggr )
= \mathrm{O}(\ln\ag_N )
\end{eqnarray*}
and similarly
$
\sup_{x\in[0,K\ag_N]}x^{d-1}h(x) = \mathrm{O}(1).
$

If $d \ge3$, the leading term is
\[
\int_{0}^{K\ag_N}x^{d-1} h(x)\,\mathrm{d}x = \mathrm{O}(\ag_N^{d-3} ).
\]

We summarize these calculations as
\[
E \Biggl\|\frac{1}{S_N}\sum_{k=1}^{S_N}X(\bs_{k,N})-\mu\Biggr\|^2
= \cases{
\mathrm{O}(\ag_N^{-1} ), & \quad if $d=1$,\cr
\mathrm{O}( \ag_N^{-2} \ln(\ag_N) ), & \quad if $d=2$,\cr
\mathrm{O}( \ag_N^{-2} ), & \quad if $d\ge3$}
\]
for Type B sampling scheme (provided the rate is slower than $N^{-1}$).
\end{example}

The last example shows that for very persistent spatial dependence, the
rate of consistency can be essentially arbitrarily slow.
\begin{example} Assume that $h(x)$ decays only at a logarithmic rate,
$
h(x)=\{\log(x\vee \mathrm{e})\}^{-1}.
$
Then, for any $d\geq1$, the left-hand side in \eqref{eq:mean_reg} is
$\ll(\log\ag_N)^{-1}$.\vadjust{\goodbreak}
\end{example}

We now turn to the case of the random design.
\begin{proposition}\label{pr:mean_rand}
Assume the random sampling design of Section~\ref{se:rd}.
If the sequence $\{\bs_{k,N}\}$ is independent of the process $X$,
and if Assumption~\ref{ass:dep} holds, then we have for any
$\varepsilon_N>0$
\[
E \Biggl\|\frac{1}{N}\sum_{k=1}^{N}X(\bs_{k,N})-\mu\Biggr\|^2
\le6 h(0)\sup_{\bs\in R_0}f^2(\bs) \eg_N^d + h(\ag_N\varepsilon
_N)+\frac{h(0)}{N}.
\]
Choosing $\varepsilon_N$ such that
$\varepsilon_N\to0$ and $\ag_N\varepsilon_N\to\infty$,
it follows that under Type \emph{B} or Type \emph{C}
sampling, the sample mean is consistent.
\end{proposition}

The bound in Proposition \ref{pr:mean_rand} can be easily applied
to any specific random sampling design and any model for the functions
$\phi_j$ in \eqref{e:xi}.
It nicely shows that what matters for the rate of consistency is the
interplay between the rate of growth of the sampling domain
and the rate of decay of dependence.

Let us explain in slightly more detail a Type C sampling situation.
Here, typically we have $\alpha_N=N^{1/d}$. Then
taking $\varepsilon_N = aN^{-1/d}\log N$, $a>0$, we see that the rate
of consistency
is $h(a\log N)\vee N^{-1}$. For typical covariance functions
$\phi_j$, like powered exponential,
Mat{\'e}rn or spherical,
$h(a\log N)$ decays faster than $N^{-1}$. In such cases, the rate of consistency
is, up to some logarithmic factor, the same as
for an i.i.d. sample. For ease of reference, we formulate
the following corollary, which can be used in practical applications.
\begin{corollary}\label{c:rand-exp}
Assume the random sampling design of Section~\ref{se:rd} with
the sequence $\{\bs_{k,N}\}$ independent the process $X$.
Suppose that $X(s)$ has representation \eqref{e:xi} and that
\eqref{eq:corr} holds with the $\phi_j$ in one of the families
specified in Example
\ref{ex:s-decay}. If Condition \refeq{cond-h} holds, and
$\ag_N \ge N^{1/d}$ then \refeq{iid-rate} holds
up to some multiplicative logarithmic factor.
\end{corollary}
%
\section{Consistency of the empirical covariance operator} \label{s:cc}
In Section \ref{s:cm}, we found the rates of consistency for the
functional sample mean. We now turn to the rates for the sample covariance
operator. Assuming the functional observations have mean zero,
the natural estimator of the covariance operator $C$ is the
sample covariance operator given by
\[
\widehat C_N=\frac{1}{N}\sum_{k=1}^NX(\bs_k)\otimes X(\bs_k).
\]
In general, the sample covariance operator is defined by
\[
\hat\Gg_N
=\frac{1}{N}\sum_{k=1}^N
\bigl( X(\bs_k)-\bar X_N \bigr)\otimes\bigl( X(\bs_k)-\bar X_N \bigr),
\]
where
\[
\bar X_N = \frac{1}{N}\sum_{k=1}^N X(\bs_k).
\]
Both operators are implemented in statistical software packages, for
example in the popular \texttt{R} package \texttt{FDA} and in a
similar MATLAB
package, see Ramsay  \textit{et al.} \cite{ramsay:hooker:graves:2009}.
The operator $\hat\Gg_N$ is used to compute the EFPC's for centered
data, while $\widehat C_N$ for data without centering.

We first derive the rates of consistency for $\widehat C_N$ assuming
$EX(\bs) = 0$. Then we turn to the operator $\hat\Gg_N$. The proofs
are obtained by applying the technique developed for the estimation of
the functional mean. It is a general approach based on the estimation
of the second moments of an appropriate norm (between estimator and
estimand) so that the conditions in Definition~\ref{def:nids} can come
into play. It is broadly applicable to all statistics obtained by
simple averaging of quantities defined at single spatial location.
The proofs are thus similar to those presented in the simplest case in
Section \ref{s:p-cm}, but the notation becomes more cumbersome because
of the increased complexity of the objects to be averaged. To
conserve space, these proofs are not included.

We begin by observing that
\begin{eqnarray*}
E \|\widehat C_N-C \|_{\cS}^2
&=&E \langle\widehat C_N-C , \widehat C_N-C\rangle_{\cS}\\
&=&\frac{1}{N^2}\sum_{k=1}^N\sum_{\ell=1}^N
E\langle X(\bs_k)\otimes X(\bs_k)-C ,
X(\bs_\ell)\otimes X(\bs_\ell)-C\rangle_{\cS}.
\end{eqnarray*}
It follows that under Assumption~\ref{ass:dep-cov}
%
\begin{equation}\label{eq:cov_consist}
E \|\widehat C_N-C \|_{\cS}^2
\leq\frac{1}{N^2}\sum_{k=1}^N\sum_{\ell=1}^NH (\|\bs_k-\bs_\ell
\|_2 ).
\end{equation}
Relation \eqref{eq:cov_consist} is used as the starting point of all
proofs, cf. the proof of Proposition \ref{p:gnis} in Section~\ref{s:cm}.
Modifying the proofs of Section~\ref{s:cm}, we arrive at the
following results.
\begin{proposition}\label{p:gnis_cov}
Let Assumption~\ref{ass:dep-cov} hold,
and assume that $\mathfrak{S}_N$ defines
a \emph{non-random} design of Type \emph{A}, \emph{B} or \emph{C}.
Then for any $\rho_N>0$
\[
E \|\widehat C_N-C \|_{\cS}^2
\leq H(\rho_N)+H(0)I_{\rho_N}(\mathfrak{S}_N).
\]
Hence under the Type \emph{B} or Type \emph{C} non-random sampling, with $\rho_N$ as
in \emph{(ii)}
of Definition~\ref{def:nids}, the empirical covariance operator is consistent.
\end{proposition}
\begin{proposition}\label{pr:mean_reg_cov}
Assume the sampling design of Section~\ref{se:nrrd}.
Let Assumption~\ref{ass:dep-cov} hold, with
some function $H$ such that
$x^{d-1}H(x)$ is monotone on $[b,\infty)$, $b>0$.
Then under Type \emph{B} sampling
\[
E \|\widehat C_N-C \|_{\cS}^2
\leq\frac{1}{\ag_N^d}
\biggl\{ d(3\Delta)^d \int_{0}^{K\ag_N}x^{d-1} H (x )\,\mathrm{d}x
+\mathrm{o}(1) \sup_{x\in[0,K\ag_N]}x^{d-1}H(x)\biggr\}
\]
for some large enough constant $K$ which is independent of $N$.
Under Type \emph{C} sampling, the factor $1/\ag_N^d$
is replaced by $\mathrm{O} (N^{-1} )$.
\end{proposition}
\begin{proposition}\label{pr:mean_rand_cov}
Assume the random sampling design of Section~\ref{se:rd}.
If the sequence $\{\bs_{k,N}\}$ is independent of the process $X$
and if Assumption~\ref{ass:dep-cov} holds,
then we have for any $\varepsilon_N>0$,
\[
E \|\widehat C_N-C \|^2
\le6 H(0)\sup_{\bs\in R_0}f^2(\bs) \eg_N^d + H(\ag_N\varepsilon
_N)+\frac{H(0)}{N}.
\]
It follows that under Type \emph{B} or Type \emph{C}
sampling the sample covariance operator is consistent.
\end{proposition}

Introducing the (unobservable) operator
\[
\tilde\Gg_N =\frac{1}{N}\sum_{k=1}^N
\bigl( X(\bs_k)-\mu\bigr)\otimes\bigl( X(\bs_k)-\mu\bigr),
\]
we see that
\[
\tilde\Gg_N - \hat\Gg_N = (\bar X_N - \mu) \otimes(\bar X_N - \mu).
\]
Therefore,
\[
E\|\hat\Gg_N-C\|_{\cS}^2\leq
2 E\|\tilde\Gg_N-C\|_{\cS}^2
+ 2 E\|(\bar X_N - \mu) \otimes(\bar X_N - \mu)\|_{\cS}^2.
\]
The bounds in Propositions \ref{p:gnis_cov}, \ref{pr:mean_reg_cov}
and \ref{pr:mean_rand_cov} apply to $E\|\tilde\Gg_N-C\|_{\cS}^2$.
Observe that
\[
E\|(\bar X_N - \mu) \otimes(\bar X_N - \mu)\|_{\cS}^2
= E\|\bar X_N - \mu\|^4.
\]
If $X(\bs)$ are bounded variables, that is, $\sup_{t\in[0,1]}|X(\bs
;t)|\leq B<\infty$ a.s., then $\|\bar X_N - \mu\|^4\leq4B^2\|\bar
X_N - \mu\|^2$.
It follows that under Assumption~\ref{ass:dep} we obtain the same
order of magnitude for the bounds of $E\|\bar X_N - \mu\|^4$
as we have obtained in Propositions \ref{p:gnis}, \ref{pr:mean_reg}
and \ref{pr:mean_rand} for $E\|\bar X_N - \mu\|^2$. In general $E\|
\bar X_N - \mu\|^4$ can neither be bounded in terms of $E\|\bar X_N -
\mu\|^2$ nor with $E\|\hat C_N - C\|^2_\mathcal{S}$. To bound fourth
order moments, conditions
on the covariance between the variables $Z_{k,\ell}:=\langle X(\bs
_{k,N})-\mu, X(\bs_{\ell,N})-\mu\rangle$ and $Z_{i,j}$ for all
$1\leq i,j,k,\ell\leq N$ are unavoidable. However,
a~simpler general approach is to
require higher order moments of $\|X(\bs)\|$.
More precisely, we notice that for any $p>1$, by the H\"older inequality,
\[
E\|\bar{X}_N-\mu\|^4\leq(E\|\bar{X}_N-\mu\|^2 )^{1/p}
\bigl(E\|\bar{X}_N-\mu\|^{(4p-2)/(p-1)} \bigr)^{(p-1)/p}.
\]
Thus as long as $E\|X(\bs)\|^{(4p-2)/(p-1)}<\infty$, we conclude
that, by stationarity,
\[
E\|\bar{X}_N-\mu\|^4\leq M(p) (E\|\bar{X}_N-\mu\|^2 )^{1/p},\vadjust{\goodbreak}
\]
where $M(p)$ depends on the distribution of $X(\bs)$ and on $p$,
but not on $N$.
It is now evident how the results of Section \ref{s:cm} can be used to obtain
bounds for $E\|\hat{\Gamma}_N-C\|_\mathcal{S}^2$. We state in Proposition
\ref{pr:cov} the
version for the general non-random design. The special cases follow, and
the random designs are treated analogously. It follows that if
Assumptions \ref{ass:dep} and \ref{ass:dep-cov} hold, then
$E\|\hat{\Gamma}_N-C\|_\mathcal{S}^2 \to0$, under Type B or C
sampling, provided $E\|X(\bs)\|^{4+\delta}<\infty$.
\begin{proposition}\label{pr:cov}
Let Assumptions~\ref{ass:dep} and \ref{ass:dep-cov} hold and assume that
for some $\delta>0$ we have $E\|X(\bs)\|^{4+\delta}<\infty$.
Assume further that $\mathfrak{S}_N$ defines
a non-random design of Type \emph{A}, \emph{B} or \emph{C}. Then for any $\rho_N>0$ we have
%
\begin{equation}\label{e:covBB}
E\|\hat{\Gamma}_N-C\|_\mathcal{S}^2\leq2 \{H(\rho_N)+H(0)I_{\rho_N}
(\mathfrak{S}_N) \}+2K(\delta) \{h(\rho_N)+h(0)I_{\rho_N}
(\mathfrak{S}_N) \}^{\delta/(2+\delta)},
\end{equation}
where $K(\delta)$ depends only on $\dg$.

If $X(\bs_1)$ is a.s. bounded by some finite constant $B$, then
we can formally let $\delta$ in \eqref{e:covBB} go to $\infty$,
with $K(\infty)=4B^2$.
\end{proposition}
%
%
\section{Inconsistent empirical functional principal components}
\label{s:example}
%
%
We define
$
X^\star= \langle X(\bzero), \cdot\rangle X(\bzero).
$
Observe that
$
X^\star( X(\bzero) ) = \|X(\bzero)\|^2 X(\bzero).
$
Thus, $\|X(\bzero)\|^2= \sum_{j=1}^\infty\xi_j^2(\bzero)$ is an
eigenvalue of
$X^\star$.
Note also that for $x\in L^2$,
\[
X^\star(x)(t) = \biggl(\int X(\bzero;u) x(u) \,\mathrm{d}u\biggr) X(\bzero;t)
= \int c^\star(t,u) x(u) \,\mathrm{d}u,
\]
where
\[
c^\star(t,u) = X(\bzero;t) X(\bzero;u).
\]
Since
\[
E \dint( c^\star(t,u))^2 \,\mathrm{d}t\,\mathrm{d}u = E\|X(\bzero)\|^4 < \infty,
\]
the operator $X^\star$ is Hilbert--Schmidt almost surely.
\begin{proposition} \label{p:example}
Suppose representation \refeq{xi} holds with stationary mean zero
Gaussian processes $\xi_j$ such that
\[
E[\xi_j(\bs) \xi_j(\bs+\bh)] = \la_j \rho_j(h),\qquad h =\| \bh\|,
\]
where each $\rho_j$ is a continuous correlation function,
and $\sum_j \la_j < \infty$. Assume the processes
$\xi_j$ and $\xi_i$ are independent if $i\neq j$.
If $\mathfrak{S}_N=\{\bs_1, \bs_2, \ldots,\bs_N\}\subset\mathbb
{R}^d$ with
$\bs_N\to\bzero$, then
%
\begin{equation}\label{e:convXstar}
\lim_{N\to\infty} E \|\widehat C_N - X^\star\|_\cS^2 = 0.
\end{equation}
\end{proposition}

Proposition \ref{p:example} is proven in Section \ref{s:p-cm}.
Since the eigenvalues of $X^\star$ are random they
cannot be close to any of the $\la_j$.
The eigenfunctions of $\widehat C_N$ are also close to random functions
in $L^2$, and do not converge to the FPC's $e_j$.

We now present a very specific example that illustrates Proposition
\ref{p:example}.
\begin{example} \label{ex:example}
Suppose
%
\begin{equation}\label{e:X-ex}
X(s; t) = \zg_1(s) e_1(t) + \sqrt\la\zg_2(s) e_2(t),
\end{equation}
where the $\zg_1$ and $\zg_2$ are i.i.d. processes on the line,
and $0<\la< 1$. Assume that the
processes $\zg_1$ and $\zg_2$ are Gaussian with mean zero and covariances
$E[\zg_j(s)\zg_j(s+h)] = \exp\{ - h^2\}, j =1, 2$.
Thus, each $Z_j := \zg_j(0)$ is standard normal. Rearranging the
terms, we obtain
\[
X^\star(x) =
\bigl( Z_1^2\langle x, e_1 \rangle+ \sqrt\la Z_1 Z_2 \langle x, e_2
\rangle\bigr) e_1
+\bigl(\sqrt\la Z_1 Z_2 \langle x, e_1 \rangle+ \la Z_2^2 \langle x, e_2
\rangle
\bigr) e_2.
\]
The matrix
\[
\left[\matrix{Z_1^2 & \sqrt\la Z_1 Z_2\cr
\sqrt\la Z_1 Z_2 & \la Z_2^2}
\right]
\]
has only one positive eigenvalue $Z_1^2 + \la Z_2^2 = \|X(0)\|^2$.
A normalized eigenfunction associated with it is
%
\begin{equation}\label{e:f-ex}
f:= \frac{X(0)}{\| X(0)\|} = [ Z_1^2 + \la Z_2^2 ]^{-1/2}
\bigl( Z_1 e_1 + \sqrt\la Z_2 e_2\bigr ).
\end{equation}
Denote by $\hat v_1$ a normalized eigenfunction corresponding to
the largest eigenvalue of $\widehat C_N$. By Lemma \ref{l:B4.3},
$\hat v_1$ is close in probability
to $\operatorname{sign} (\langle\hv_1, f \rangle) f$. It is thus not
close to $\operatorname{sign} (\langle\hv_1, e_1 \rangle) e_1$.
%
%
%
\begin{figure}

\includegraphics{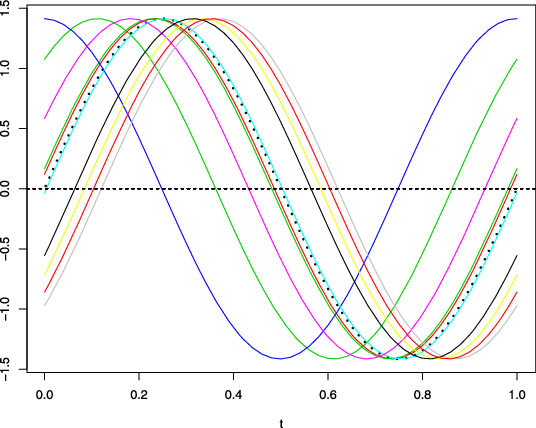}

\caption{Ten simulated EFPC's $\hat v_1$ for process \protect\refeq{X-ex}
with $\la=0.5$ and
$e_1(t) = \sqrt{2}\sin(2\uppi t)$, $e_2(t) = \sqrt{2}\cos(2\uppi t)$
($N=100$). The dotted line is the population eigenfunction.}\label{f:tenXi}
\end{figure}
%

Ten simulated $\hv_1$ are shown in Figure \ref{f:tenXi}.
To compute each $\hv_1$, we simulated curves $X(s_n)$ given by \refeq{X-ex}
with with $e_1(t) = \sqrt{2}\sin(2\uppi  t),
e_2(t) = \sqrt{2}\cos(2\uppi  t)$,
$\la=0.5$. We set $s_n=n^{-1}, n=1,2,\ldots, N, N=100$.
The random vectors $[\zg_j(s_1), \zg_j(s_2), \ldots, \zg_j(s_N)]^T$
were generated using the R function
\texttt{rmvnorm} which uses the singular value decomposition to
simulate Gaussian vectors with predetermined covariances
(in our case,
$\operatorname{Cov}(\zg_j(s_k), \zg_j(s_\ell)) = \exp\{-|k^{-1}
- \ell^{-1}|\}$).

The EFPC $\hv_1$ is a linear combination of $e_1$ and $e_2$ with
random weights.
As formula \refeq{f-ex} suggests, the function $e_1$ is likely to receive
a larger weight. The weights, and so the simulated $\hv_1$,
cluster because both $Z_1$ and $Z_2$ are standard normal.
\end{example}

We now state a general result showing that Type A
sampling generally leads to inconsistent estimators if the
spatial dependence does not vanish.
\begin{proposition}\label{pr:dep_inc}
Assume that $E\langle X(\bs_{1})-\mu, X(\bs_{2})-\mu\rangle\geq
b(\|\bs_1-\bs_2\|_2)>0$, where $b(x)$ is non-increasing.
Then under Type \emph{A} sampling
the sample mean $\bar X_N$ is not a consistent estimator of $\mu$.
Similarly, if
$EX(\bs)=0$ and
%
\begin{equation}\label{e:inc-A}
E\langle X(\bs_{1})\otimes X(\bs_1)-C
, X(\bs_{2})\otimes X(\bs_2)-C \rangle_{\cS}\geq
B(\|\bs_1-\bs_2\|_2)>0,
\end{equation}
where $B(x)$ is nonincreasing, then under Type \emph{A} sampling
the sample covariance $\widehat C_N$ is not a consistent estimator of $C$.
\end{proposition}

We illustrate Proposition \ref{pr:dep_inc} with a further example that
complements Proposition~\ref{p:example}
in a sense that in Proposition~\ref{p:example} the functional model
was complex, but the spatial distribution of the $\bs_k$ simple.
In Example \ref{ex:gauss2}, we allow a general Type A distribution,
but consider the simple model~\refeq{X-ex}.
\begin{example} \label{ex:gauss2}
We focus on condition \refeq{inc-A} for the FPC's. For the general
model \refeq{xi}, the left-hand side of \refeq{inc-A} is equal to
\[
\kg(\bs_1, \bs_2) =
\sum_{i,j \ge1} \operatorname{Cov}(\xi_i(\bs_1)\xi_j(\bs_1) ,
\xi_i(\bs_2)\xi_j(\bs_2) ).
\]
If the processes $\xi_j$ satisfy the assumptions of Proposition
\ref{p:example}, then, by Lemma \ref{l:cov2},
\[
\operatorname{Cov}(\xi_i(\bs_1)\xi_j(\bs_1) , \xi_i(\bs_2)\xi
_j(\bs_2) )
= \la_i^2 r_i + \la_j^2 r_j + \la_i\la_j \frac{r_i + r_j}{2}
- (\la_i^{3/2} r_i + \la_j^{3/2} r_j )\sqrt{\la_i + \la_j},
\]
where $r_i = \rho_i(\| \bs_1 - \bs_2 \|)$.

To calculate $\kg(\bs_1, \bs_2)$ in a simple case, corresponding to
\refeq{X-ex}, suppose
%
\begin{equation}\label{e:inc-xi}
\la_1=1,\qquad \la_2 = \la, \qquad 0 < \la< 1,\qquad \la_i = 0,\qquad i >2
\quad\mbox{and}\quad \rho_1 = \rho_2 = \rho.
\end{equation}
Then,
\[
\kg(\bs_1, \bs_2) = f(\la) \rho(\| \bs_1 - \bs_2 \|),
\]
where
\[
f(\la) =\bigl(3-2\sqrt{2}\bigr) (1 + \la^2)
+ 2 [1 + \la+ \la^2 - (1 + \la^{3/2}) ( 1+\la)^{1/2}].
\]
The function $f$ increases from about 0.17 at $\la=0$ to about 0.69 at
$\la=1$.

We have verified that if the functional random field \refeq{xi}
satisfies the assumptions of Proposition~\ref{p:example} and \refeq{inc-xi},
then $\widehat C_N$ is an inconsistent estimator of $C$
under Type A sampling,
whenever $\rho(h)$ is a nonincreasing function of $h$.
\end{example}
%
\section{\texorpdfstring{Proofs of the results of Sections \protect\ref{s:cm}, \protect\ref{s:cc} and \protect\ref{s:example}}
{Proofs of the results of Sections 3, 4 and 5}} \label{s:p-cm}
We will use the following well-known lemma.
\begin{lemma} \label{l:cov2}
Suppose $X$ and $Y$ are jointly normal mean zero random variables
such that
$
EX^2 = \sg^2, EY^2 = \nu^2, E[XY]=\rho\sg\nu.
$
Then
\[
\operatorname{Cov}(X^2, Y^2) = 2\rho^2\sg^2\nu^2.
\]
\end{lemma}
\begin{pf*}{Proof of Proposition~\ref{p:gnis}}
By Assumption~\ref{ass:dep}, we have
\begin{eqnarray*}
&&E \Biggl\|\frac{1}{N}\sum_{k=1}^{N}X(\bs_{k,N})-\mu\Biggr\|^2\\
&&\quad =
\frac{1}{N^2}\sum_{k=1}^{N}\sum_{\ell=1}^{N}
E\langle X(\bs_{k,N})-\mu, X(\bs_{\ell,N})-\mu
\rangle\\
&&\quad \leq\frac{1}{N^2}\sum_{k=1}^{N}\sum_{\ell=1}^{N}
h (\|\bs_{k,N}-\bs_{\ell,N}\|_2 )\\
&&\quad \leq\frac{1}{N^2}\sum_{k=1}^N\sum_{\ell=1}^N \bigl(
h(\rho_N)I\{\|\bs_{k,N}-\bs_{\ell,N}\|_2\geq\rho_N\}
+h(0)I\{\|\bs_{k,N}-\bs_{\ell,N}\|_2\leq\rho_N\} \bigr)\\
&&\quad \leq h(\rho_N)+h(0) I_{\rho_N}(\mathfrak{S}_N).
\end{eqnarray*}
%
\upqed\end{pf*}

The following lemma is a simple calculus problem and will be used in the
proof of Proposition~\ref{pr:mean_reg}.
\begin{lemma}\label{le:help}
Assume that $f$ is a nonnegative function which is monotone on $[0,b]$
and on $[b,\infty)$.
Then
\[
\sum_{k=0}^Lf \biggl(\frac{k}{N}\biggr )\frac{1}{N}\leq
\int_{0}^{L/N}f(x)\,\mathrm{d}x+\frac{2}{N}\sup_{x\in[0,L/N]}|f(x)|.
\]
\end{lemma}
\begin{pf*}{Proof of Proposition~\ref{pr:mean_reg}}
By Assumption~\ref{ass:dep},
\[
E \Biggl\|\frac{1}{S_N}\sum_{k=1}^{S_N}X(\bs_{k,N})-\mu\Biggr\|^2
\leq\frac
{1}{S_N^2}\sum_{k=1}^{S_N}\sum_{\ell=1}^{S_N}
h (\|\bs_{k,N}-\bs_{\ell,N}\|_2 ).
\]
Let $\mathbf{a}=(a_1,\ldots,a_d)$ and $\mathbf{b}=(b_1,\ldots,b_d)$
be two elements on $\mathcal{Z}(\boldsymbol{\delta})$.
We define
$d(\mathbf{a},\mathbf{b})=\min_{1\leq i\leq d}v_i(\mathbf
{a},\mathbf{b})$,
where $v_i(\mathbf{a},\mathbf{b})$ is the number of edges between
$a_i$ and $b_i$.
For any two points $\bs_{k,N}$ and $\bs_{\ell,N}$, we have
%
\begin{equation}\label{dist}
d(\bs_{k,N},\bs_{\ell,N})=m \qquad\mbox{from some } m\in\{0,\ldots
,KN^{1/d}\},
\end{equation}
where $K$ depends on $\operatorname{diam}(R_0)$.
It is easy to see that the number
of points on the grid having
distance $m$ from a given point is less than $2d(2m+1)^d$, $m\geq0$.
Hence, the number of pairs for which \eqref{dist} holds
is $< 2d(2m+1)^{d-1}N$.
On the other hand, if $d(\bs_{k,N},\bs_{\ell,N})=m$,
then $\|\bs_{k,N}-\bs_{\ell,N}\|_2\geq m\delta_0 \eta_N$. Let
us assume without loss of generality that $\delta_0=1$.
Noting that there is no loss of
generality if we assume that $x^{\delta-1}h(x)$
is also monotone on $[0,b]$, we obtain
by Lemma~\ref{le:help} for large enough $N$ and $K<K'<K''$
\begin{eqnarray*}
&&\frac{1}{S_N^2}\sum_{k=1}^{S_N}\sum_{\ell=1}^{S_N}
h (\|\bs_{k,N}-\bs_{\ell,n}\|_2 )\\
&&\quad \leq2d\sum_{m=1}^{K'N^{1/d}}\frac{(2m+1)^{d-1}}{N}
h (m\eta_N )+\frac{2h(0)}{N}\\
&&\quad \leq2d \biggl(\frac{3}{\eta
_N} \biggr)^{d-1}
\sum_{m=0}^{K'N^{1/d}+1} \biggl(\frac{m}{N} N\eta_N \biggr)^{d-1}
h \biggl(\frac{m}{N} N\eta_N \biggr)\frac{1}{N}+\frac{2h(0)}{N}\\
&&\quad \leq2d \biggl(\frac{3}{\eta_N}\biggr )^{d-1}
\biggl(\int_{0}^{K''N^{1/d-1}}(N\eta_Nx)^{d-1}
h (N\eta_Nx )\,\mathrm{d}x \\
&&\hphantom{\quad \leq2d \biggl(\frac{3}{\eta_N}\biggr )^{d-1}
\biggl(}{} + \frac{2}{N}\sup_{x\in[0,K''\ag_N/\Delta]}x^{d-1}h(x) \biggr)
+\frac{2h(0)}{N}\\
&&\quad = \frac{(3\Delta)^d d}{\ag_N^d}\int_{0}^{K''\ag_N/\Delta}x^{d-1}
h (x )\,\mathrm{d}x\\
&&\qquad{} +\frac{4d(3\Delta)^{d-1}}{\ag_N^{d-1}N^{1/d}}
\sup_{x\in[0,K''\ag_N/\Delta]}x^{d-1}h(x)+\frac
{2h(0)}{N}.
\end{eqnarray*}
By Lemma~\ref{le:scaling}, Type B sampling implies $\ag_N\to\infty$ and
$\ag_N=\mathrm{o} (N^{1/d} )$. This shows \eqref{eq:mean_reg}. Under Type C
sampling $1/\ag_N^d\ll1/N$. The proof is finished.
\end{pf*}
\begin{pf*}{Proof of Proposition~\ref{pr:mean_rand}}
This time we have
\begin{eqnarray*}
E \Biggl\|\frac{1}{N}\sum_{k=1}^N X(\bs_{k,N})-\mu\Biggr\|^2&\leq&\frac
{1}{N^2}\sum_{k=1}^N\sum_{\ell=1}^N
Eh (\|\bs_{k,N}-\bs_{\ell,N}\|_2 )\\
& \leq&
\ag_N^{-2d}\int_{R_N}\int_{R_N}h (\|\bs-\br\|_2 )f(\ag_N^{-1}\bs)
f(\ag_N^{-1}\br) \,\mathrm{d}\bs \,\mathrm{d}\br+\frac{h(0)}{N}\\
& =&
\int_{R_0}\int_{R_0}h (\ag_N\|\bs-\br\|_2 )f(\bs)
f(\br) \,\mathrm{d}\bs \,\mathrm{d}\br+\frac{h(0)}{N}.
\end{eqnarray*}
Furthermore, for any $\varepsilon_N>0$,
\begin{eqnarray*}
&&\int_{R_0}\int_{R_0}h (\ag_N\|\bs-\br\|_2 )f(\bs)
f(\br) \,\mathrm{d}\bs \,\mathrm{d}\br\\
&&\quad \leq
h(0)\int_{R_0}\int_{R_0}f(\bs)
f(\br)I \{\|\bs-\br\|_2\leq\varepsilon_N \} \,\mathrm{d}\bs \,\mathrm{d}\br+h(\ag
_N\varepsilon_N)\\
& &\quad\leq
h(0)\sup_{\bs\in R_0}f^2(\bs)\times\int_{R_0}\int_{R_0}I \{ \|\bs
-\br\|_2\leq\varepsilon_N \}\,\mathrm{d}\bs \,\mathrm{d}\br+h(\ag_N\varepsilon_N).
\end{eqnarray*}
Now for fixed $\br$ it is not difficult to show that
$\int_{R_0}I \{ \|\bs-\br\|_2\leq\varepsilon_N \}\,\mathrm{d}\bs\leq6
\varepsilon_N^d$.
(The constant 6 could be replaced with $\uppi ^{d/2}/\Gamma(d/2+1))$.
\end{pf*}
\begin{pf*}{Proof of Proposition~\ref{p:example}} Observe that
\[
\|\widehat C_N - X^\star\|_\cS^2
= \dint\Biggl\{\frac{1}{N} \sum_{n=1}^N
[ X(\bs_n; t) X(\bs_n;u) - X(\bzero; t) X(\bzero;u)]\Biggr\}^2\,\mathrm{d}t\,\mathrm{d}u.
\]
Therefore,
\[
\|\widehat C_N - X^\star\|_\cS^2 \le2 I_1(N) + 2I_2(N),
\]
where
\[
I_1(N) = \dint\Biggl\{\frac{1}{N} \sum_{n=1}^N
X(\bs_n; t) \bigl(X(\bs_n; u) - X(\bzero;u)\bigr) \Biggr\}^2\,\mathrm{d}t\,\mathrm{d}u
\]
and
\[
I_2(N) = \dint\Biggl\{\frac{1}{N} \sum_{n=1}^N
X(\bzero;u) \bigl(X(\bs_n; t) - X(\bzero;t)\bigr) \Biggr\}^2\,\mathrm{d}t\,\mathrm{d}u.
\]

We will show that $E I_1(N) \to0$. The argument for $I_2(N)$ is
the same. Observe that
\begin{eqnarray*}
I_1(N)
&=& \frac{1}{N^2} \sum_{k, \ell=1}^N
\dint X(\bs_k; t) \bigl(X(\bs_k; u) - X(\bzero;u)\bigr)
X(\bs_\ell; t) \bigl(X(\bs_\ell; u) - X(\bzero;u)\bigr) \,\mathrm{d}t\,\mathrm{d}u
\\
&=& \frac{1}{N^2} \sum_{k, \ell=1}^N
\int X(\bs_k; t) X(\bs_\ell; t) \,\mathrm{d}t
\int\bigl(X(\bs_k; u) - X(\bzero;u)\bigr)\bigl(X(\bs_\ell; u) - X(\bzero;u)\bigr)\,\mathrm{d}u.
\end{eqnarray*}
Thus,
\[
E I_1(N) \le\frac{1}{N^2} \sum_{k, \ell=1}^N
\biggl\{ E\biggl(\int X(\bs_k; t) X(\bs_\ell; t) \,\mathrm{d}t\biggr)^2\biggr\}^{1/2}
\biggl\{ E\biggl(\int Y_k(u) Y_\ell(u) \,\mathrm{d}u \biggr)^2\biggr\}^{1/2},
\]
where
\[
Y_k(u) = X(\bs_k; u) - X(\bzero;u).
\]
We first deal with the integration over $t$:
\begin{eqnarray*}
E\biggl(\int X(\bs_k; t) X(\bs_\ell; t) \,\mathrm{d}t\biggr)^2
&\le& E \int X^2(\bs_k; t)\,\mathrm{d}t\int X^2(\bs_\ell; t)\,\mathrm{d}t
= E [\| X(\bs_k)\|^2 \| X(\bs_\ell)\|^2 ]\\
&\le&\{ E \| X(\bs_k)\|^4 \}^{1/2} \{ E \| X(\bs_\ell)\|^4
\}^{1/2}
= E\| X(\bzero)\|^4.
\end{eqnarray*}
We thus see that
\begin{eqnarray*}
E I_1(N)
&\le&\{ E\| X(\bzero)\|^4\}^{1/2} \frac{1}{N^2} \sum_{k, \ell=1}^N
\biggl\{ E\biggl(\int Y_k(u) Y_\ell(u) \,\mathrm{d}u\biggr )^2\biggr\}^{1/2}
\\
&\le&\{ E\| X(\bzero)\|^4 \}^{1/2} \frac{1}{N^2} \sum_{k, \ell=1}^N
\biggl\{ E\biggl(\int Y_k^2 (u)\,\mathrm{d}u\biggr)^2 \biggr\}^{1/4}
\biggl\{ E\biggl(\int Y_\ell^2 (u)\,\mathrm{d}u\biggr)^2 \biggr\}^{1/4}
\\
&=& \{ E\| X(\bzero)\|^4 \}^{1/2}
\Biggl[\frac{1}{N} \sum_{k=1}^N
\biggl\{ E\biggl(\int Y_k^2 (u)\,\mathrm{d}u\biggr)^2 \biggr\}^{1/4}
\Biggr]^2.
\end{eqnarray*}
Consequently, to complete the verification of \refeq{convXstar}, it
suffices to show that
\[
\lim_{N\to\infty} \frac{1}{N} \sum_{k=1}^N
\biggl\{ E\biggl(\int Y_k^2 (u)\,\mathrm{d}u\biggr)^2 \biggr\}^{1/4} = 0.
\]
The above relation will follow from
%
\begin{equation}\label{e:di1}
\lim_{k\to\infty} E\biggl(\int Y_k^2 (u)\,\mathrm{d}u\biggr)^2 = 0.
\end{equation}
To verify \refeq{di1}, first notice that, by the orthonormality of the
$e_j$,
\[
\int Y_k^2 (u)\,\mathrm{d}u = \sum_{j=1}^\infty\bigl(\xi_j(\bs_k) - \xi
_j(\bzero)\bigr)^2.
\]
Therefore, by the independence of the processes $\xi_j$,
\begin{eqnarray*}
E\biggl(\int Y_k^2 (u)\,\mathrm{d}u\biggr)^2
&=& \sum_{j=1}^\infty E\bigl(\xi_j(\bs_k) - \xi_j(\bzero)\bigr)^4\\
&&{} + \sum_{i\neq j}
E\bigl(\xi_i(\bs_k) - \xi_i(\bzero)\bigr)^2
E\bigl(\xi_j(\bs_k) - \xi_j(\bzero)\bigr)^2.
\end{eqnarray*}
The covariance structure was specified so that
\[
E\bigl(\xi_j(\bs_k) - \xi_j(\bzero)\bigr)^2 = 2\la_j\bigl(1 -\rho_j(\|\bs
_k\|)\bigr),
\]
so the normality yields
\begin{eqnarray*}
E\biggl(\int Y_k^2 (u)\,\mathrm{d}u\biggr)^2
&\le&12 \sum_{j=1}^\infty\la_j^2 \bigl(1 -\rho_j(\|\bs_k\|)\bigr)^2\\
&&{} + 4 \Biggl\{ \sum_{j=1}^\infty\la_j \bigl(1 -\rho_j(\|\bs_k\|)\bigr)\Biggr\}^2.
\end{eqnarray*}
The right-hand side tends to zero by the Dominated Convergence theorem.
This establishes \refeq{di1}, and completes the proof of
\refeq{convXstar}.
\end{pf*}
\begin{pf*}{Proof of Proposition \ref{pr:dep_inc}}
We only check inconsistency of the sample mean. In view of the
proof of Proposition \ref{p:gnis}, we have now the lower bound
\begin{eqnarray*}
E \Biggl\|\frac{1}{N}\sum_{k=1}^{N}X(\bs_{k,N})-\mu\Biggr\|^2
&\geq&
\frac{1}{N^2}\sum_{k=1}^{N}\sum_{\ell=1}^{N}
b(\|\bs_{k,N}-\bs_{\ell,N}\|_2)\\
&\geq& b(\rho)I_\rho^2(\mathfrak{S}_N),
\end{eqnarray*}
which is by assumption bounded away from zero for $N\to\infty$.
\end{pf*}

\section*{Acknowledgements}
The research was partially supported by
NSF Grants DMS-08-04165 and DMS-09-31948 at Utah State University,
by the Banque National de Belgique and
Communaut\'e fran\c{c}aise de Belgique -- Actions de Recherche
Concert\'ees (2010--2015). We thank O. Gromenko and X. Zhang
for performing the numerical simulations reported in this paper. We
also would like to
thank two referees and the associated editor for giving clear
guidelines to improve the presentation of this paper.

\begin{supplement}
\stitle{Supplement to ``Consistency of the mean and the
principal components of spatially distributed
functional data''}
\slink[doi]{10.3150/12-BEJ418SUPP}
\sdatatype{.pdf}
\sfilename{BEJ418\_supp.pdf}
\sdescription{We provide additional examples, some remarks concerning
Assumption~\ref{ass:dep-cov} and the regular sampling design as well as a remark on the
sharpness of our bounds.}
\end{supplement}


\printhistory


\begin{thebibliography}{31}

\bibitem{bel:2011}
\begin{barticle}[mr]
\bauthor{\bsnm{Bel},~\bfnm{Liliane}\binits{L.}},
  \bauthor{\bsnm{Bar-Hen},~\bfnm{Avner}\binits{A.}},
  \bauthor{\bsnm{Petit},~\bfnm{R{\'e}my}\binits{R.}} \AND
  \bauthor{\bsnm{Cheddadi},~\bfnm{Rachid}\binits{R.}}
(\byear{2011}).
\btitle{Spatio-temporal functional regression on paleoecological data}.
\bjournal{J. Appl. Stat.}
\bvolume{38}
\bpages{695--704}.
\bid{doi={10.1080/02664760903563650}, issn={0266-4763}, mr={2773575}}
\bptok{imsref}%
\end{barticle}
\endbibitem

\bibitem{benko:hardle:kneip:2009}
\begin{barticle}[mr]
\bauthor{\bsnm{Benko},~\bfnm{Michal}\binits{M.}},
  \bauthor{\bsnm{H{\"a}rdle},~\bfnm{Wolfgang}\binits{W.}} \AND
  \bauthor{\bsnm{Kneip},~\bfnm{Alois}\binits{A.}}
(\byear{2009}).
\btitle{Common functional principal components}.
\bjournal{Ann. Statist.}
\bvolume{37}
\bpages{1--34}.
\bid{doi={10.1214/07-AOS516}, issn={0090-5364}, mr={2488343}}
\bptok{imsref}%
\end{barticle}
\endbibitem

\bibitem{bolthausen:1982}
\begin{barticle}[mr]
\bauthor{\bsnm{Bolthausen},~\bfnm{E.}\binits{E.}}
(\byear{1982}).
\btitle{On the central limit theorem for stationary mixing random fields}.
\bjournal{Ann. Probab.}
\bvolume{10}
\bpages{1047--1050}.
\bid{issn={0091-1798}, mr={0672305}}
\bptok{imsref}%
\end{barticle}
\endbibitem

\bibitem{bosq:2000}
\begin{bbook}[mr]
\bauthor{\bsnm{Bosq},~\bfnm{D.}\binits{D.}}
(\byear{2000}).
\btitle{Linear Processes in Function Spaces: Theory and Applications}.
\bseries{Lecture Notes in Statistics}
\bvolume{149}.
\baddress{New York}: \bpublisher{Springer}.
\bid{mr={1783138}}
\bptok{imsref}%
\end{bbook}
\endbibitem

\bibitem{cressie:1993}
\begin{bbook}[mr]
\bauthor{\bsnm{Cressie},~\bfnm{Noel A.~C.}\binits{N.A.C.}}
(\byear{1993}).
\btitle{Statistics for Spatial Data}.
\bseries{Wiley Series in Probability and Mathematical Statistics: Applied
  Probability and Statistics}.
\baddress{New York}: \bpublisher{Wiley}.
\bnote{Revised reprint of the 1991 edition, A~Wiley-Interscience Publication}.
\bid{mr={1239641}}
\bptok{imsref}%
\end{bbook}
\endbibitem

\bibitem{delicado:2010}
\begin{barticle}[mr]
\bauthor{\bsnm{Delicado},~\bfnm{P.}\binits{P.}},
  \bauthor{\bsnm{Giraldo},~\bfnm{R.}\binits{R.}},
  \bauthor{\bsnm{Comas},~\bfnm{C.}\binits{C.}} \AND
  \bauthor{\bsnm{Mateu},~\bfnm{J.}\binits{J.}}
(\byear{2010}).
\btitle{Statistics for spatial functional data: Some recent contributions}.
\bjournal{Environmetrics}
\bvolume{21}
\bpages{224--239}.
\bid{doi={10.1002/env.1003}, issn={1180-4009}, mr={2842240}}
\bptok{imsref}%
\end{barticle}
\endbibitem

\bibitem{du:zhang:mandrekar:2009}
\begin{barticle}[mr]
\bauthor{\bsnm{Du},~\bfnm{Juan}\binits{J.}},
  \bauthor{\bsnm{Zhang},~\bfnm{Hao}\binits{H.}} \AND
  \bauthor{\bsnm{Mandrekar},~\bfnm{V.~S.}\binits{V.S.}}
(\byear{2009}).
\btitle{Fixed-domain asymptotic properties of tapered maximum likelihood
  estimators}.
\bjournal{Ann. Statist.}
\bvolume{37}
\bpages{3330--3361}.
\bid{doi={10.1214/08-AOS676}, issn={0090-5364}, mr={2549562}}
\bptok{imsref}%
\end{barticle}
\endbibitem

\bibitem{einmahl:ruymgaart:1987}
\begin{barticle}[mr]
\bauthor{\bsnm{Einmahl},~\bfnm{J.~H.~J.}\binits{J.H.J.}} \AND
  \bauthor{\bsnm{Ruymgaart},~\bfnm{F.~H.}\binits{F.H.}}
(\byear{1987}).
\btitle{The order of magnitude of the moments of the modulus of continuity of
  multiparameter {P}oisson and empirical processes}.
\bjournal{J.~Multivariate Anal.}
\bvolume{21}
\bpages{263--273}.
\bid{doi={10.1016/0047-259X(87)90005-4}, issn={0047-259X}, mr={0884100}}
\bptok{imsref}%
\end{barticle}
\endbibitem

\bibitem{gabrys:horvath:kokoszka:2010}
\begin{barticle}[mr]
\bauthor{\bsnm{Gabrys},~\bfnm{Robertas}\binits{R.}},
  \bauthor{\bsnm{Horv{\'a}th},~\bfnm{Lajos}\binits{L.}} \AND
  \bauthor{\bsnm{Kokoszka},~\bfnm{Piotr}\binits{P.}}
(\byear{2010}).
\btitle{Tests for error correlation in the functional linear model}.
\bjournal{J. Amer. Statist. Assoc.}
\bvolume{105}
\bpages{1113--1125}.
\bid{doi={10.1198/jasa.2010.tm09794}, issn={0162-1459}, mr={2752607}}
\bptok{imsref}%
\end{barticle}
\endbibitem

\bibitem{gabrys:kokoszka:2007}
\begin{barticle}[mr]
\bauthor{\bsnm{Gabrys},~\bfnm{Robertas}\binits{R.}} \AND
  \bauthor{\bsnm{Kokoszka},~\bfnm{Piotr}\binits{P.}}
(\byear{2007}).
\btitle{Portmanteau test of independence for functional observations}.
\bjournal{J.~Amer. Statist. Assoc.}
\bvolume{102}
\bpages{1338--1348}.
\bid{doi={10.1198/016214507000001111}, issn={0162-1459}, mr={2412554}}
\bptok{imsref}%
\end{barticle}
\endbibitem

\bibitem{gohberg:1990}
\begin{bmisc}[auto:STB|2012/03/21|07:41:58]
\bauthor{\bsnm{Gohberg},~\bfnm{I.}\binits{I.}},
  \bauthor{\bsnm{Golberg},~\bfnm{S.}\binits{S.}} \AND
  \bauthor{\bsnm{Kaashoek},~\bfnm{M.~A.}\binits{M.A.}}
(\byear{1990}).
\bhowpublished{\emph{Classes of Linear Operators. Operator Theory: Advances and
  Applications} \textbf{49}. Basel: Birkha\"user}.
\bptok{imsref}%
\end{bmisc}
\endbibitem

\bibitem{hall:h-n:2006}
\begin{barticle}[mr]
\bauthor{\bsnm{Hall},~\bfnm{Peter}\binits{P.}} \AND
  \bauthor{\bsnm{Hosseini-Nasab},~\bfnm{Mohammad}\binits{M.}}
(\byear{2006}).
\btitle{On properties of functional principal components analysis}.
\bjournal{J.~R. Stat. Soc. Ser. B Stat. Methodol.}
\bvolume{68}
\bpages{109--126}.
\bid{doi={10.1111/j.1467-9868.2005.00535.x}, issn={1369-7412}, mr={2212577}}
\bptok{imsref}%
\end{barticle}
\endbibitem

\bibitem{hormann:kokoszka:2010}
\begin{barticle}[mr]
\bauthor{\bsnm{H{\"o}rmann},~\bfnm{Siegfried}\binits{S.}} \AND
  \bauthor{\bsnm{Kokoszka},~\bfnm{Piotr}\binits{P.}}
(\byear{2010}).
\btitle{Weakly dependent functional data}.
\bjournal{Ann. Statist.}
\bvolume{38}
\bpages{1845--1884}.
\bid{doi={10.1214/09-AOS768}, issn={0090-5364}, mr={2662361}}
\bptok{imsref}%
\end{barticle}
\endbibitem

\bibitem{suppl:hormann:kokoszka:2011}
\begin{bmisc}[auto:STB|2012/03/21|07:41:58]
\bauthor{\bsnm{H{\"o}rmann},~\bfnm{S.}\binits{S.}} \AND
  \bauthor{\bsnm{Kokoszka},~\bfnm{P.}\binits{P.}}
(\byear{2012}).
\bhowpublished{Supplement to ``Consistency of the mean and the principal
  components of spatially distributed functional data.''
  DOI:\doiurl{10.3150/12-BEJ418SUPP}.}
\bptok{imsref}%
\end{bmisc}
\endbibitem

\bibitem{jenish:prucha:2009}
\begin{barticle}[mr]
\bauthor{\bsnm{Jenish},~\bfnm{Nazgul}\binits{N.}} \AND
  \bauthor{\bsnm{Prucha},~\bfnm{Ingmar~R.}\binits{I.R.}}
(\byear{2009}).
\btitle{Central limit theorems and uniform laws of large numbers for arrays of
  random fields}.
\bjournal{J. Econometrics}
\bvolume{150}
\bpages{86--98}.
\bid{doi={10.1016/j.jeconom.2009.02.009}, issn={0304-4076}, mr={2525996}}
\bptok{imsref}%
\end{barticle}
\endbibitem

\bibitem{jiang:wang:2010}
\begin{barticle}[mr]
\bauthor{\bsnm{Jiang},~\bfnm{Ci-Ren}\binits{C.R.}} \AND
  \bauthor{\bsnm{Wang},~\bfnm{Jane-Ling}\binits{J.L.}}
(\byear{2010}).
\btitle{Covariate adjusted functional principal components analysis for
  longitudinal data}.
\bjournal{Ann. Statist.}
\bvolume{38}
\bpages{1194--1226}.
\bid{doi={10.1214/09-AOS742}, issn={0090-5364}, mr={2604710}}
\bptok{imsref}%
\end{barticle}
\endbibitem

\bibitem{lahiri:1996}
\begin{barticle}[mr]
\bauthor{\bsnm{Lahiri},~\bfnm{Soumendra~Nath}\binits{S.N.}}
(\byear{1996}).
\btitle{On inconsistency of estimators based on spatial data under infill
  asymptotics}.
\bjournal{Sankhy\=a Ser. A}
\bvolume{58}
\bpages{403--417}.
\bid{issn={0581-572X}, mr={1659130}}
\bptok{imsref}%
\end{barticle}
\endbibitem

\bibitem{lahiri:2003}
\begin{barticle}[mr]
\bauthor{\bsnm{Lahiri},~\bfnm{S.~N.}\binits{S.N.}}
(\byear{2003}).
\btitle{Central limit theorems for weighted sums of a spatial process under a
  class of stochastic and fixed designs}.
\bjournal{Sankhy\=a}
\bvolume{65}
\bpages{356--388}.
\bid{issn={0972-7671}, mr={2028905}}
\bptok{imsref}%
\end{barticle}
\endbibitem

\bibitem{lahiri:zhu:2006}
\begin{barticle}[mr]
\bauthor{\bsnm{Lahiri},~\bfnm{S.~N.}\binits{S.N.}} \AND
  \bauthor{\bsnm{Zhu},~\bfnm{Jun}\binits{J.}}
(\byear{2006}).
\btitle{Resampling methods for spatial regression models under a class of
  stochastic designs}.
\bjournal{Ann. Statist.}
\bvolume{34}
\bpages{1774--1813}.
\bid{doi={10.1214/009053606000000551}, issn={0090-5364}, mr={2283717}}
\bptok{imsref}%
\end{barticle}
\endbibitem

\bibitem{loh:2005}
\begin{barticle}[mr]
\bauthor{\bsnm{Loh},~\bfnm{Wei-Liem}\binits{W.L.}}
(\byear{2005}).
\btitle{Fixed-domain asymptotics for a subclass of {M}at\'ern-type {G}aussian
  random fields}.
\bjournal{Ann. Statist.}
\bvolume{33}
\bpages{2344--2394}.
\bid{doi={10.1214/009053605000000516}, issn={0090-5364}, mr={2211089}}
\bptok{imsref}%
\end{barticle}
\endbibitem

\bibitem{park:kim:park:hwang:2009}
\begin{barticle}[mr]
\bauthor{\bsnm{Park},~\bfnm{Byeong~U.}\binits{B.U.}},
  \bauthor{\bsnm{Kim},~\bfnm{Tae~Yoon}\binits{T.Y.}},
  \bauthor{\bsnm{Park},~\bfnm{Jeong-Soo}\binits{J.S.}} \AND
  \bauthor{\bsnm{Hwang},~\bfnm{S.~Y.}\binits{S.Y.}}
(\byear{2009}).
\btitle{Practically applicable central limit theorem for spatial statistics}.
\bjournal{Math. Geosci.}
\bvolume{41}
\bpages{555--569}.
\bid{doi={10.1007/s11004-008-9184-2}, issn={1874-8961}, mr={2516125}}
\bptok{imsref}%
\end{barticle}
\endbibitem

\bibitem{paul:peng:2009}
\begin{barticle}[mr]
\bauthor{\bsnm{Paul},~\bfnm{Debashis}\binits{D.}} \AND
  \bauthor{\bsnm{Peng},~\bfnm{Jie}\binits{J.}}
(\byear{2009}).
\btitle{Consistency of restricted maximum likelihood estimators of principal
  components}.
\bjournal{Ann. Statist.}
\bvolume{37}
\bpages{1229--1271}.
\bid{doi={10.1214/08-AOS608}, issn={0090-5364}, mr={2509073}}
\bptok{imsref}%
\end{barticle}
\endbibitem

\bibitem{ramsay:hooker:graves:2009}
\begin{bmisc}[auto:STB|2012/03/21|07:41:58]
\bauthor{\bsnm{Ramsay},~\bfnm{J.}\binits{J.}},
  \bauthor{\bsnm{Hooker},~\bfnm{G.}\binits{G.}} \AND
  \bauthor{\bsnm{Graves},~\bfnm{S.}\binits{S.}}
(\byear{2009}).
\bhowpublished{\emph{Functional Data Analysis with R and MATLAB}. New York: Springer}.
\bptok{imsref}%
\end{bmisc}
\endbibitem

\bibitem{ramsay:silverman:2005}
\begin{bbook}[mr]
\bauthor{\bsnm{Ramsay},~\bfnm{J.~O.}\binits{J.O.}} \AND
  \bauthor{\bsnm{Silverman},~\bfnm{B.~W.}\binits{B.W.}}
(\byear{2005}).
\btitle{Functional Data Analysis},
\bedition{2nd} ed.
\bseries{Springer Series in Statistics}.
\baddress{New York}: \bpublisher{Springer}.
\bid{mr={2168993}}
\bptok{imsref}%
\end{bbook}
\endbibitem

\bibitem{reiss:ogden:2007}
\begin{barticle}[mr]
\bauthor{\bsnm{Reiss},~\bfnm{Philip~T.}\binits{P.T.}} \AND
  \bauthor{\bsnm{Ogden},~\bfnm{R.~Todd}\binits{R.T.}}
(\byear{2007}).
\btitle{Functional principal component regression and functional partial least
  squares}.
\bjournal{J. Amer. Statist. Assoc.}
\bvolume{102}
\bpages{984--996}.
\bid{doi={10.1198/016214507000000527}, issn={0162-1459}, mr={2411660}}
\bptok{imsref}%
\end{barticle}
\endbibitem

\bibitem{rio:1993}
\begin{barticle}[mr]
\bauthor{\bsnm{Rio},~\bfnm{Emmanuel}\binits{E.}}
(\byear{1993}).
\btitle{Covariance inequalities for strongly mixing processes}.
\bjournal{Ann. Inst. Henri Poincar\'e Probab. Stat.}
\bvolume{29}
\bpages{587--597}.
\bid{issn={0246-0203}, mr={1251142}}
\bptok{imsref}%
\end{barticle}
\endbibitem

\bibitem{stein:1999}
\begin{bbook}[mr]
\bauthor{\bsnm{Stein},~\bfnm{Michael~L.}\binits{M.L.}}
(\byear{1999}).
\btitle{Interpolation of Spatial Data: Some Theory for Kriging}.
\bseries{Springer Series in Statistics}.
\baddress{New York}: \bpublisher{Springer}.
\bid{mr={1697409}}
\bptok{imsref}%
\end{bbook}
\endbibitem

\bibitem{stute:1984}
\begin{barticle}[mr]
\bauthor{\bsnm{Stute},~\bfnm{Winfried}\binits{W.}}
(\byear{1984}).
\btitle{The oscillation behavior of empirical processes: The multivariate
  case}.
\bjournal{Ann. Probab.}
\bvolume{12}
\bpages{361--379}.
\bid{issn={0091-1798}, mr={0735843}}
\bptok{imsref}%
\end{barticle}
\endbibitem

\bibitem{yamanishi:tanaka:2003}
\begin{barticle}[mr]
\bauthor{\bsnm{Yamanishi},~\bfnm{Yoshihiro}\binits{Y.}} \AND
  \bauthor{\bsnm{Tanaka},~\bfnm{Yutaka}\binits{Y.}}
(\byear{2003}).
\btitle{Geographically weighted functional multiple regression analysis: A
  numerical investigation}.
\bjournal{J. Japanese Soc. Comput. Statist.}
\bvolume{15}
\bpages{307--317}.
\bid{issn={0915-2350}, mr={2027947}}
\bptok{imsref}%
\end{barticle}
\endbibitem

\bibitem{zhang:2004}
\begin{barticle}[mr]
\bauthor{\bsnm{Zhang},~\bfnm{Hao}\binits{H.}}
(\byear{2004}).
\btitle{Inconsistent estimation and asymptotically equal interpolations in
  model-based geostatistics}.
\bjournal{J. Amer. Statist. Assoc.}
\bvolume{99}
\bpages{250--261}.
\bid{doi={10.1198/016214504000000241}, issn={0162-1459}, mr={2054303}}
\bptok{imsref}%
\end{barticle}
\endbibitem

\end{thebibliography}
\end{document}